\documentclass{amsart}

\usepackage{amsfonts,amssymb}
\usepackage{amsmath}

\newtheorem{thm}{Theorem}
\newtheorem{lem}[thm]{Lemma}
\newtheorem{cor}[thm]{Corollary}
\newtheorem{prop}[thm]{Proposition}
\newtheorem{conj}[thm]{Conjecture}
\newtheorem{defn}{Definition}

\begin{document}

\title{On algebraic Riccati equations
associated with  {\it M\/}-Matrices}

\author{Chun-Hua Guo}
\address{Department of Mathematics and Statistics, 
University of Regina, Regina,
SK S4S 0A2, Canada} 
\email{chun-hua.guo@uregina.ca}

\thanks{This work was supported in part by a grant from
the Natural Sciences and Engineering Research Council of Canada.}



\begin{abstract}

We consider the algebraic Riccati equation for which
the four coefficient matrices form an $M$-matrix $K$. 
When $K$ is a nonsingular $M$-matrix or an irreducible singular $M$-matrix, 
the Riccati equation is known to have a minimal nonnegative solution 
and several efficient methods are available to find this solution. 
In this paper we are mainly interested in the case where $K$ is a reducible singular 
$M$-matrix. Under a regularity assumption on the $M$-matrix $K$, we show that the Riccati equation still has a minimal nonnegative solution. 
We also study the properties of this particular solution 
and explain how the solution can be found by existing methods. 

\end{abstract}

\maketitle

\section{Introduction} 

We consider the algebraic Riccati equation
\begin{equation}\label{NARE}
XCX-XD-AX+B=0,
\end{equation}
where $A, B, C, D$ are real matrices of sizes $m\times m, m\times n,
n\times m, n\times n$, respectively, and 
\begin{align}\label{Mm}
K=\left [ \begin{array}{cc} D & -C\\ -B & A \end{array} \right ]
\end{align}
is an $M$-matrix. 

For any matrices $A, B\in {\mathbb R}^{p \times q}$, we write $A\ge B\ (A>B)$ if
$a_{ij}\ge b_{ij} (a_{ij}>b_{ij})$ for all $i,j$. A real square matrix $A$ is called a
$Z$-matrix if all its off-diagonal elements are nonpositive. Note that any $Z$-matrix $A$
can be written as $sI-B$ with $B\ge 0$. A $Z$-matrix $A$ is called an $M$-matrix if
$s\ge\rho(B)$, where $\rho(\cdot)$ is the spectral radius; it is a singular $M$-matrix if
$s=\rho(B)$ and a nonsingular $M$-matrix if $s>\rho(B)$.

If $K$ is an arbitrary $M$-matrix, then the equation \eqref{NARE} may not have a solution. 
A simple example is the reducible singular $M$-matrix 
\begin{align}\label{km}
K=\left [\begin{array}{cc}
0 & 0\\
-1 & 0
\end{array} \right ], 
\end{align}
for which the corresponding scalar equation \eqref{NARE} does not have any solutions. 

It is proved in \cite{guo01} that the equation \eqref{NARE} always has a minimal nonnegative solution if $K$  in \eqref{Mm} is a nonsingular 
$M$-matrix or an irreducible singular $M$-matrix. This assumption on the matrix $K$ has now become a standard assumption. 
In \cite{xxl12}, the equation \eqref{NARE} with $K$ being a nonsingular $M$-matrix or an irreducible singular $M$-matrix is formally
called an $M$-matrix algebraic Riccati equation. The case where $K$ is a reducible singular $M$-matrix has never been studied, 
due to some technical difficulties associated with it. However, this case is of both theoretical and practical interest. 
The most important application of the Riccati equation \eqref{NARE} is in the study of Markov chains \cite{roge94}, where $K=-Q$ and $Q$ is the generator of a 
Markov chain. So all diagonal entries of $Q$ are nonpositive, all off-diagonal entries of $Q$ are nonnegative, and every row sum of $Q$ is nonpositve. 
In this case, $K$ is an $M$-matrix with $Ke\ge 0$, where $e$ is the vector of ones. Of particular interest is the case where $K$ is a singular 
$M$-matrix with $Ke=0$. In the study of Markov chains, the irreducibility of $Q$ is often assumed since $Q$ is irreducible for most 
interesting Markov chains.  But it is still desirable to have results available when $Q$ happens to be reducible. In general, even if $K$ is an irreducible singular $M$-matrix, it may still be close to some reducible singular 
$M$-matrix. If we run into difficulties for a reducible singular $M$-matrix, we would also run into difficulties for nearby irreducible singular 
$M$-matrices. It is therefore of importance to study the case where $K$ is a reducible singular $M$-matrix.

\section{Existence  of a minimal nonnegative solution}

As we have seen earlier,  the equation \eqref{NARE} may not have a solution for an arbitrary $M$-matrix $K$. So to guarantee the existence 
of a solution we still need to add some assumption on $K$. The assumption we need is that  
 $Kv\ge 0$ for some vector $v>0$.  Note that this assumption is always satisfied if $K$ is from the study of Markov chains, as we mentioned above, whether it is reducible or not. We also note that the so-called $M$-splittings of a reducible singular $M$-matrix have been studied in \cite{schn84}, 
where this same assumption on the singular $M$-matrix also plays an important role. 

For convenience, we introduce the following definition. 

\begin{defn}
An $M$-matrix $A$ is said to be regular if $Av\ge 0$ for some $v>0$.
\end{defn}
 
It is well known \cite{bepl94} that for any nonsingular $M$-matrix $A$ there is a vector $v>0$ such that $Av>0$ 
and that for any irreducible singular $M$-matrix $A$ there is a vector $v>0$ such that $Av=0$. 
Therefore, nonsingular $M$-matrices and irreducible singular $M$-matrices are always regular $M$-matrices. 
It is easily seen that the reducible singular $M$-matrix in \eqref{km} is not a regular $M$-matrix. 
Note that  any $Z$-matrix $A$ such that $Av\ge 0$ for some $v>0$ is an $M$-matrix 
\cite{bepl94} and hence a regular $M$-matrix by definition. Note also that any principal submatrix of a regular $M$-matrix 
is still a regular $M$-matrix. 

We are going to prove that the equation \eqref{NARE} has a minimal nonnegative solution whenever $K$ is a regular $M$-matrix. 
We start with the following result. 

\begin{lem}\label{suff}
Suppose that the matrix $K$ in \eqref{Mm} is  a regular $M$-matrix 
and all diagonal entries of $A$  are positive.  
Then \eqref{NARE}  has a nonnegative solution $\Phi$ such that $D-C\Phi$ is a regular $M$-matrix. 
Moreover, $\Phi\le X$ for any nonnegative solution of the inequality ${\mathcal R}(X)\le 0$, where 
\[
{\mathcal R}(X)=XCX-XD-AX+B. 
\]
 In particular, 
$\Phi$ is the minimal nonnegative solution of the equation \eqref{NARE}. 
\end{lem}

\begin{proof}
The proof is an adaptation of the proof of \cite[Theorem 3.1]{guo01}. 
We write $A=A_1-A_2$ and $D=D_1-D_2$, where $A_1={\rm diag}\,(A)$ and $D_1={\rm diag}\,(D)$. 
We take $X_0=0$ and use the fixed-point iteration:
\begin{equation}\label{fp1}
A_1X_{i+1}+X_{i+1}D_1=X_iCX_i+X_iD_2+A_2X_i+B, \quad i=0,1,\ldots.
\end{equation}
The iteration is well defined since the diagonal entries of $D$ are nonnegative and the diagonal entries of $A$ are 
assumed to be positive. 
It is easily shown by induction that $X_i\le X_{i+1}$ for any $i\ge 0$.

Since $K$ is a regular $M$-matrix, we have $Kv\ge 0$ for some $v>0$. Write $v=[v_1^T \; v_2^T]^T$ with 
$v_1\in {\mathbb R}^n$ and $v_2\in {\mathbb R}^m$.  Then we have 
\begin{equation}\label{teqs}
D_1v_1-D_2v_1-Cv_2\ge 0, \quad A_1v_2-A_2v_2-Bv_1\ge 0.
\end{equation}
We now show that $X_kv_1\le v_2$ for all $k\ge 0$.
The inequality is trivial for $k=0$.  
Assume that $X_iv_1\le v_2\, (i\ge 0)$. Then,
by (\ref{fp1}) and (\ref{teqs}),
\begin{eqnarray*}
A_1X_{i+1}v_1+X_{i+1}D_1v_1&=&X_iCX_iv_1+X_iD_2v_1+A_2X_iv_1+Bv_1\\
&\le &X_iCv_2+X_iD_2v_1+A_2v_2+Bv_1\\
&\le &X_iD_1v_1+A_1v_2.
\end{eqnarray*}
It follows that $A_1X_{i+1}v_1\le A_1v_2+(X_i-X_{i+1})D_1v_1\le A_1v_2$.
So $X_{i+1}v_1\le v_2$. Thus, we have proved
by induction that $X_kv_1\le v_2$ for all $k\ge 0$.
Now, the sequence $\{X_i\}$
is monotonically increasing and bounded above, and hence has a limit.
Let $\Phi=\lim_{i\rightarrow \infty}X_i$.
Then $\Phi$ is a nonnegative solution of \eqref{NARE} by \eqref{fp1}. We also have $\Phi v_1\le v_2$ and then 
$(D-C\Phi)v_1\ge Dv_1-Cv_2\ge 0$. Thus, the $Z$-matrix $D-C\Phi$ is a regular $M$-matrix.  
We now let $X$ be any nonnegative solution of ${\mathcal R}(X)\le 0$ and re-examine the iteration 
\eqref{fp1}. 
Since $XCX+XD_2+A_2X+B\le A_1X+XD_1$ by ${\mathcal R}(X)\le 0$, 
it is easily shown by induction that $X_i\le X$ for any $i\ge 0$. 
It follows that $\Phi\le X$. In particular,  $\Phi$ is the minimal nonnegative solution of \eqref{NARE}. 
 \end{proof}

\begin{thm}\label{thmn}
If the matrix $K$ in \eqref{Mm} is  a regular $M$-matrix 
then $(\ref{NARE})$ has a minimal nonnegative solution $\Phi$ and 
$D-C\Phi$ is a regular $M$-matrix.
\end{thm}

\begin{proof}
If all diagonal entries of $A$ are positive, the result is already proved. 
Suppose that $r$ diagonal entries of $A$ are $0$. Then the $r$ rows of $K$ containing these diagonal entries 
must be zero rows since $K$ is a regular $M$-matrix. If $r=m$, then $A=0$ and $B=0$, and the equation \eqref{NARE} has a minimal nonnegative solution $\Phi=0$ and $D-C\Phi$ is a regular $M$-matrix. So we assume $1\le r<m$. Then, there is an $m\times m$ permutation matrix $P$ 
such that 
\[
PAP^T=\left [\begin{array}{cc}
\widetilde{A} & \widehat{A}\\
0 & 0
\end{array} \right ], \quad PB=\left [\begin{array}{c}
\widetilde{B} \\
0 
\end{array} \right ], \quad CP^T=\left [\begin{array}{cc}
\widetilde{C}  & \widehat{C}\\
\end{array} \right ], 
\]
where $\widetilde{A}$ is $(m-r)\times (m-r)$, $\widetilde{B}$ is $(m-r)\times n$, and $\widetilde{C}$ is $n\times (m-r)$. 
The equation \eqref{NARE} is equivalent to 
\[
(PX)(CP^T)(PX)-(PX)D-(PAP^T)(PX)+PB=0.
\]
Let 
\[
PX=\left [\begin{array}{c}
\widetilde{X} \\
\widehat{X}
\end{array} \right ], 
\]
where $\widetilde{X}$ is $(m-r)\times n$. 
Then 
\begin{align}\label{eqn99}
\left [\begin{array}{c}
\widetilde{X} \\
\widehat{X}
\end{array} \right ]\left [\begin{array}{cc}
\widetilde{C}  & \widehat{C}\\
\end{array} \right ]\left [\begin{array}{c}
\widetilde{X} \\
\widehat{X}
\end{array} \right ]-\left [\begin{array}{c}
\widetilde{X} \\
\widehat{X}
\end{array} \right ]D-\left [\begin{array}{cc}
\widetilde{A} & \widehat{A}\\
0 & 0
\end{array} \right ]\left [\begin{array}{c}
\widetilde{X} \\
\widehat{X}
\end{array} \right ]+\left [\begin{array}{c}
\widetilde{B} \\
0 
\end{array} \right ]=0.
\end{align}
We now take $\widehat{X}=0$. Then the above equation reduces to 
\begin{align}\label{eqnred}
\widetilde{X}\widetilde{C}\widetilde{X}-\widetilde{X}D-\widetilde{A}\widetilde{X}+\widetilde{B}=0.
\end{align}
Since the matrix 
\begin{align}\label{MKt}
\widetilde{K}=\left [ \begin{array}{cc}
D & -\widetilde{C}\\
-\widetilde{B} & \widetilde{A}
\end{array} \right ]
\end{align}
is a principal submatrix of the regular $M$-matrix 
\[
\left [ \begin{array}{cc}
I_n & 0\\
0 & P
\end{array} \right ]
\left [ \begin{array}{cc}
D & -C\\
-B & A
\end{array} \right ]\left [ \begin{array}{cc}
I_n & 0\\
0 & P
\end{array} \right ]^T,
\]
it is also a regular $M$-matrix. Since all diagonal entries of $\widetilde{A}$ are positive, we know from Lemma  \ref{suff} that 
the equation \eqref{eqnred} has a minimal nonnegative solution $\widetilde{\Phi}$ and $D-\widetilde{C}\widetilde{\Phi}$ is a regular $M$-matrix. 
The solution 
\begin{align}\label{eqnss}
\Phi=P^T \left [\begin{array}{c}
\widetilde{\Phi} \\
0
\end{array} \right ]
\end{align}
of \eqref{NARE} is then the only candidate for the minimal nonnegative solution. To confirm that a minimal nonnegative solution exists, 
we let 
\[
P^T \left [\begin{array}{c}
\widetilde{X} \\
\widehat{X}
\end{array} \right ]
\] 
be any nonnegative solution of \eqref{NARE}.  Then  we have by \eqref{eqn99} that 
\[
\widetilde{X}\widetilde{C}\widetilde{X}+\widetilde{X}\widehat{C}\widehat{X}-\widetilde{X}D-\widetilde{A}\widetilde{X}-\widehat{A}\widehat{X}+\widetilde{B}=0.
\]
Thus 
\[
\widetilde{X}\widetilde{C}\widetilde{X}-\widetilde{X}D-\widetilde{A}\widetilde{X}+\widetilde{B}
=-\widetilde{X}\widehat{C}\widehat{X}+\widehat{A}\widehat{X}\le 0.
\]
It follows from Lemma  \ref{suff} that $\widetilde{\Phi}\le \widetilde{X}$. Thus the matrix $\Phi$ in \eqref{eqnss} is indeed the minimal 
nonnegative solution of \eqref{NARE}. 
Note that $D-C\Phi=D-(CP^T)(P\Phi)=D-\widetilde{C}\widetilde{\Phi}$ is a regular $M$-matrix. 
  \end{proof}

Associated with the matrix $K$ in \eqref{Mm} is the matrix 
\begin{equation}\label{Hm}
H=\left [\begin{array}{cc}I_n & 0\\
0 & -I_m
\end{array} \right ] K=\left [ \begin{array}{cc}
D & -C\\
B & -A
\end{array} \right ]. 
\end{equation}
The next factorization result follows from Theorem \ref{thmn} easily. 

\begin{cor}\label{whf}
If the matrix $K$ in \eqref{Mm}  is a regular $M$-matrix, then there exist nonnegative matrices $\Phi$ and $\Psi$
such that
\begin{equation}\label{wh}
\left [\begin{array}{cc}
D & -C\\
B & -A
\end{array} \right ] \left [\begin{array}{cc}
I_n & \Psi\\
\Phi & I_m
\end{array} \right ]=\left [\begin{array}{cc}
I_n & \Psi\\
\Phi & I_m
\end{array} \right ]\left [\begin{array}{cc}
R & 0\\
0 & -S
\end{array} \right ],
\end{equation}
where $R=D-C\Phi$ and $S=A-B\Psi$ are regular $M$-matrices.
\end{cor}

\begin{proof}
By Theorem \ref{thmn}, \eqref{NARE} has a minimal nonnegative
solution $\Phi$ and $D-C\Phi$ is a regular $M$-matrix. Thus   
\begin{equation}\label{f1}
\left [\begin{array}{cc}
D & -C\\
B & -A
\end{array} \right ] \left [\begin{array}{c}
I_n \\
\Phi
\end{array} \right ]=\left [\begin{array}{c}
I_n \\
\Phi
\end{array} \right ] R.
\end{equation}
for $R=D-C\Phi$. 
Since $K$ is a regular $M$-matrix, we have $Kv\ge 0$ for some $v=[v_1^T \; v_2^T]^T>0$, where 
$v_1\in {\mathbb R}^n$ and $v_2\in {\mathbb R}^m$. Then 
\[
\widehat{K}=\left [\begin{array}{cc}
A & -B\\
-C & D
\end{array} \right ]
\]
is a $Z$-matrix such that 
\[
\widehat{K}\left [\begin{array}{c}v_2\\v_1
\end{array}\right ]\ge 0. 
\]
So $\widehat{K}$ is also a regular $M$-matrix. 
Theorem \ref{thmn} implies that the dual equation of \eqref{NARE} 
\begin{equation}\label{nare2}
YBY-YA-DY+C=0
\end{equation}
has a minimal nonnegative solution $\Psi$ such that
$A-B\Psi$ is a regular $M$-matrix.  Thus 
\begin{equation}\label{f2}
\left [\begin{array}{cc}
D & -C\\
B & -A
\end{array} \right ] \left [\begin{array}{c}
\Psi \\
I_m
\end{array} \right ]=\left [\begin{array}{c}
\Psi \\
I_m
\end{array} \right ](-S)
\end{equation}
for $S=A-B\Psi$.
The factorization \eqref{wh} is obtained by combining \eqref{f1} and \eqref{f2}.      \end{proof}

\section{Properties of the minimal nonnegative solution}

We assume that the matrix $K$ in \eqref{Mm} is  a regular $M$-matrix. So 
the equation \eqref{NARE} has a minimal nonnegative solution $\Phi$ and 
$D-C\Phi$ is a regular $M$-matrix, and the dual equation \eqref{nare2} has a minimal nonnegative solution $\Psi$ and 
$A-B\Psi$ is a regular $M$-matrix. 
In this section we prove some additional properties of $\Phi$ and $\Psi$. 

\begin{prop}\label{pr}
Suppose that $K$ is a regular $M$-matrix and $Kv\ge 0$ for $v=[v_1^T \; v_2^T]^T>0$, where 
$v_1\in {\mathbb R}^n$ and $v_2\in {\mathbb R}^m$. Then 
$\Phi v_1\le v_2$ and $\Psi v_2\le v_1$. Moreover, $I_m-\Phi \Psi$ and $I_n-\Psi \Phi$ are both regular $M$-matrices. 
\end{prop}

\begin{proof}
If $A$ has no zero diagonal entries, we already know from the proof of Lemma \ref{suff} that 
$\Phi v_1\le v_2$. Suppose that $r$ diagonal entries of $A$ are $0$. If $r=m$, then $A=0$ and $B=0$. In this case,  $\Phi=0$ 
and $\Phi v_1\le v_2$ is trivial. So we assume $1\le r<m$ and use the notations in the proof of Theorem \ref{thmn}. 
Let $Pv_2=\left [\begin{array}{c}\widetilde{v}_2\\
\widehat{v}_2 \end{array} \right ]$, where $\widetilde{v}_2\in {\mathbb R}^{m-r}$. 
Then the matrix $\widetilde{K}$ in \eqref{MKt} satisfies $\widetilde{K} \left [\begin{array}{c}v_1\\
\widetilde{v}_2
\end{array} \right ]\ge 0$. It follows that $\widetilde{\Phi}v_1\le \widetilde{v}_2$. 
Then 
\[
\Phi v_1=P^T\left [\begin{array}{c}
\widetilde{\Phi} \\
0
\end{array} \right ]v_1\le P^T\left [\begin{array}{c}
\widetilde{v}_2  \\
0
\end{array} \right ]\le P^T\left [\begin{array}{c}
\widetilde{v}_2 \\
\widehat{v}_2
\end{array} \right ]=v_2.
\]
By applying the result $\Phi v_1\le v_2$ to the dual equation \eqref{nare2}, we get $\Psi v_2\le v_1$. 
We then have $\Phi \Psi v_2\le \Phi v_1\le v_2$ and thus $(I_m-\Phi\Psi)v_2\ge 0$. So 
$I_m-\Phi\Psi$ is a regular $M$-matrix. Similarly, $I_n-\Psi\Phi$ is a regular $M$-matrix. 
 \end{proof}

Note that the relation in \eqref{wh} does not always give a similarity transformation since the matrix 
\begin{align}\label{PP}
\left [\begin{array}{cc}
I_n & \Psi\\
\Phi & I_m
\end{array} \right ]
\end{align}
may be singular. The matrix \eqref{PP} is nonsingular if and only if the $M$-matrix $I_m-\Phi \Psi$ is nonsingular, i.e., 
if and only if  $\rho(\Phi \Psi)<1$.  

We denote by $\mathbb C_<$, $\mathbb C_{\le}$, $\mathbb C_>$, and $\mathbb C_{\ge}$ the open left half plane, 
the closed left half plane, the open right half plane, and the closed right half plane, respectively. 
We have the following information about the eigenvalues of $H$ even if the matrix \eqref{PP} is singular. 

\begin{lem}\label{lemev}
For any $M$-matrix $K$, 
the matrix $H$ in \eqref{Hm} has $m_1 (\le m)$ eigenvalues in $\mathbb C_<$, $n_1 (\le n)$ eigenvalues in $\mathbb C_>$, and 
$r=(m-m_1)+(n-n_1)$ zero eigenvalues. 
\end{lem}

\begin{proof}
For any $\epsilon>0$, $K_{\epsilon}=K+\epsilon I_{m+n}$ is a nonsingular $M$-matrix. 
Thus (see \cite{guo01}) 
\[
H_{\epsilon}=\left [ \begin{array}{cc}
D+\epsilon I_n & -C\\
B & -(A+\epsilon I_m)
\end{array} \right ]
\]
has $m$ eigenvalues in $\mathbb C_<$ and $n$ eigenvalues in $\mathbb C_>$. 
The result follows by letting $\epsilon\to 0^+$ since eigenvalues are continuous functions of $\epsilon$. 
  \end{proof}

Since 
\begin{equation}\label{inv}
H\left [\begin{array}{c}
I_n \\
X
\end{array} \right ]=\left [\begin{array}{c}
I_n \\
X
\end{array} \right ] (D-CX)
\end{equation}
for any solution $X$ of \eqref{NARE}, the column space of $\left [\begin{array}{c}
I_n \\
X
\end{array} \right ]$ is an invariant subspace of $H$ corresponding to $n$ of its eigenvalues. 
For the minimal nonnegative solution $\Phi$ of \eqref{NARE} (when $K$ in \eqref{Mm} is a regular $M$-matrix),  $D-C\Phi$ is an $M$-matrix 
and thus  all its eigenvalues are in ${\mathbb C}_{\ge}$. 
When $H$ has exactly $n$ eigenvalues in ${\mathbb C}_{\ge}$, the column space of $\left [\begin{array}{c}
I_n \\
\Phi
\end{array} \right ]$ must be the invariant subspace of $H$ corresponding to these eigenvalues. 
When $H$ has more than $n$ eigenvalues in ${\mathbb C}_{\ge}$, the next result shows that the column space of $\left [\begin{array}{c}
I_n \\
\Phi
\end{array} \right ]$ is an invariant subspace of $H$ corresponding to its $n$ eigenvalues with the largest real parts.

\begin{thm}\label{thmev}
 Assume that the matrix $K$ in \eqref{Mm} is  a regular $M$-matrix. 
Let all eigenvlaues of $H$ in \eqref{Hm} be arranged in an descending order by their real parts, and be denoted by $\lambda_1, \ldots, \lambda_n, \lambda_{n+1}, \ldots, \lambda_{n+m}$. Then the eigenvalues of $D-C\Phi$ are $\lambda_1, \ldots, \lambda_n$ and 
the eigenvalues of $A-B\Psi$ are $-\lambda_{n+1}, \ldots, -\lambda_{n+m}$. In particular, $\lambda_n$ and $\lambda_{n+1}$ are 
real numbers. 
\end{thm}

\begin{proof}
For any $\epsilon>0$, let $K_{\epsilon}$ and 
$H_{\epsilon}$ be as in the proof of Lemma \ref{lemev}. 
Then the equation \eqref{NARE} corresponding to 
$K_{\epsilon}$ has a minimal nonnegative solution $\Phi_{\epsilon}$ and the eigenvalues of $D+\epsilon I_n-C\Phi_{\epsilon}$ 
are the $n$ eigenvalues of $H_{\epsilon}$ in $\mathbb C_>$. 
Since $K$ is a regular $M$-matrix, $Kv\ge 0$ for some $v>0$ partitioned as in Proposition \ref{pr}. Then  we have $K_{\epsilon}v>0$ 
and $\Phi_{\epsilon}v_1\le v_2$. So the set $\{\Phi_{\epsilon}\}$ is bounded and thus there is a sequence 
$\{\epsilon_k\}$ with $\lim_{k\to \infty}\epsilon_k=0$ such that $\lim_{k\to \infty}\Phi_{\epsilon_k}=
\widehat{\Phi}$ exists. It is clear that $\widehat{\Phi}$ is a nonnegative solution of the equation \eqref{NARE} (corresponding to the matrix $K$) and that the eigenvalues of 
$D-C\widehat{\Phi}$ are $\lambda_1, \ldots, \lambda_n$. Since $\Phi\le \widehat{\Phi}$, we have ${\rm trace}(D-C\Phi)\ge {\rm trace} 
(D-C\widehat{\Phi})$. However, the trace of a matrix is equal to the sum of its eigenvalues. 
It follows that the eigenvalues of $D-C\Phi$ are $\lambda_1, \ldots, \lambda_n$. 
By applying this result to the dual equation \eqref{nare2}, we know that the eigenvalues of $A-B\Psi$ are $-\lambda_{n+1}, \ldots, -\lambda_{n+m}$. The eigenvalues $\lambda_n$ and $\lambda_{n+1}$ are 
real numbers because for any $M$-matrix the eigenvalue with smallest real part must be a real number. 
 \end{proof}

\begin{cor}\label{sim} 
Let $K$ and $H$ be as in Theorem \ref{thmev}. Then 
the eigenvalues of $A-\Phi C$
are  $-\lambda_{n+1}, \ldots, -\lambda_{n+m}$ and the eigenvalues of $D-\Psi B$ are $\lambda_1, \ldots, \lambda_n$. 
In particular, $D-\Psi B$ is nonsingular if and only if $D-C\Phi$ is nonsingular, and $A-\Phi C$ is nonsingular if and only if $A-B\Psi$ is nonsingular.
\end{cor}

\begin{proof}
Since $\Phi$ is a solution of  \eqref{NARE}, it is easily verified (and is well known) that 
\begin{equation}\label{eqnH}
\left [\begin{array}{cc}
I_n & 0\\
\Phi  & I_m
\end{array} \right ]^{-1} H \left [\begin{array}{cc}
I_n & 0\\
\Phi & I_m
\end{array} \right ] =
\left [\begin{array}{cc}
D-C\Phi & -C\\
0 & -(A-\Phi C)
\end{array} \right ].
\end{equation}
By this similarity transformation, the eigenvalues of $A-\Phi C$
are  $-\lambda_{n+1}, \ldots, -\lambda_{n+m}$ in view of Theorem \ref{thmev}. 
Applying this result to the dual equation \eqref{nare2}, we know that 
the eigenvalues of $D-\Psi B$ are $\lambda_1, \ldots, \lambda_n$. 
 \end{proof}

When $K$ in \eqref{Mm} is an irreducible singular $M$-matrix, the minimal nonnegative solutions $\Phi$ and $\Psi$ are positive \cite{guo02} and 
the matrix $H$ has a simple zero eigenvalue except for a critical case where $H$ has a double zero eigenvalue with only 
one linearly independent eigenvector \cite{guo01}. 
Our emphasis in this paper is on regular $M$-matrices that are singular and reducible. 
We now assume that $H$ has only one linearly independent eigenvector corresponding to the 
zero eigenvalue of multiplicity $r\ge 1$. Without this assumption, the minimal nonnegative solution of 
\eqref{NARE} may not be a continuous function of $K$ on the set of regular $M$-matrices. 
A trivial example for $m=n=1$ is given by  
\[
K=\left [\begin{array}{cc}
0 & 0\\
0 & 0
\end{array} \right ], 
\quad  K_{\epsilon}=\left [\begin{array}{cc}
\epsilon & -\epsilon\\
-\epsilon  & \epsilon
\end{array} \right ]. 
\]
The minimal nonnegative solution corresponding to $K$ is $0$, but the minimal nonnegative solution corresponding to $K_{\epsilon}$ is $1$ for any $\epsilon>0$. 
A less trivial example for $m=n=2$ is 
\[
K=\left [\begin{array}{cccc}
1 & 0 & -1 & 0\\
0 & 0 & 0 & 0\\
-1 & 0 & 1 & 0\\
0 & 0 & 0 & 0
\end{array} \right ], 
\quad  K_{\epsilon}=\left [\begin{array}{cccc}
1 & 0 & -1 & 0\\
0 & \epsilon & 0 & -\epsilon\\
-1 & 0 & 1 & 0\\
0 &  -\epsilon  & 0 & \epsilon
\end{array} \right ]. 
\]
The minimal nonnegative solution corresponding to $K$ is $\left [\begin{array}{cc}
1 & 0\\
0 & 0 
\end{array} \right ]$, but the minimal nonnegative solution corresponding to $K_{\epsilon}$ is $\left [\begin{array}{cc}
1 & 0\\
0 & 1 
\end{array} \right ]$ for any $\epsilon>0$. 

For the rest of this section, we assume that $H$ has only one linearly independent eigenvector corresponding to the 
zero eigenvalue. In other words, we assume that the null space of $H$ is one-dimensional. In view of the relation \eqref{Hm}, the null spaces of $K$ and $K^T$ are then both one-dimensional. Note that $K$ can have at most one zero diagonal entry under this  assumption.  
By the Perron--Frobenius theorem \cite{bepl94}, there are nonnegative nonzero vectors 
$\left[\begin{array}{c}
u_1\\
u_2
\end{array}\right ]$ and $\left[\begin{array}{c}
v_1\\
v_2
\end{array}\right ]$, where $u_1, v_1\in {\mathbb R}^n$ and $u_2, v_2\in {\mathbb R}^m$, 
 such that 
\begin{equation}\label{uuvv}
K\left[\begin{array}{c}
v_1\\
v_2
\end{array}\right ]=0, \quad [u_1^T\;u_2^T] K=0.  
\end{equation}
They are each unique up to a scalar multiple,  in view of our assumption on the null spaces. 

\begin{lem}\label{lem8}
Let $u_1, u_2, v_1, v_2$ be as in \eqref{uuvv}. Then 
 $u_1^Tv_1\ne u_2^Tv_2$ if and only if zero is a simple eigenvalue of $H$. 
\end{lem}

\begin{proof}
Let $P$ be a nonsingular matrix such that 
\begin{equation}\label{H}
H=P\left [\begin{array}{cc} 
J & 0\\
0 & W
\end{array} \right ] P^{-1}, 
\end{equation}
where $J$ is the $r\times r$ Jordan block associated with the zero eigenvalue and $W$ is nonsingular.  
It follows from \eqref{Hm} and \eqref{uuvv} that 
\begin{equation}\label{uv}
 [u_1^T\;-u_2^T] H=0, \quad
H\left[\begin{array}{c}
v_1\\
v_2
\end{array}\right ]=0. 
\end{equation}
Let $e_i$ be the $i$th column of the identity matrix $I_{m+n}$. Then by \eqref{H} and \eqref{uv}
\begin{equation}\label{uuuvvv}
[u_1^T\;-u_2^T] P=k_1e_r^T,   \quad  P^{-1}\left[\begin{array}{c}
v_1\\
v_2
\end{array}\right ]=k_2e_1, 
\end{equation}
where $k_1$ and $k_2$ are nonzero constants. Multiplying the two equations in \eqref{uuuvvv} gives 
$u_1^Tv_1-u_2^Tv_2=k_1k_2 e_r^Te_1$. Thus  $u_1^Tv_1\ne u_2^Tv_2$ if and only if $r=1$. 
 \end{proof} 

We now assume that the singular $M$-matrix $K$ is regular. Then we have $w>0$ such that $Kw\ge 0$, 
but we do not necessarily have $Kv=0$ for some $v>0$. A simple example is 
\[
K=\left [\begin{array}{ccc}
1 & 0 & 0\\
0 & 1 & -1\\
0 & -1 & 1
\end{array}
\right ]. 
\]
So the vectors $u_1, u_2, v_1, v_2$  in \eqref{uuvv}  are nonnegative, but not necessarily positive. 
However, we still have the following result. 

\begin{lem}\label{lem9}
Let $K$ be a regular singular $M$-matrix, and let  $u_1, u_2, v_1, v_2$ be as in \eqref{uuvv}.
Then $\Phi v_1\le v_2$, $\Psi v_2\le v_1$, $u_2^T\Phi \le u_1^T$, $u_1^T\Psi \le u_2^T$. 
\end{lem}

\begin{proof}
If $A$ has no zero diagonal entries, we re-examine the proof of Lemma \ref{suff} and find that we still have 
$X_i v_1\le v_2$ for all $i\ge 0$, for the new vectors $v_1$ and $v_2$ here. Since $\lim_{i\to \infty}X_i=\Phi$ was proved there, 
we have $\Phi v_1\le v_2$. When $A$ has zero diagonal entries, we can still show that $\Phi v_1\le v_2$ using the procedure in the 
proof of Proposition \ref{pr}.

By taking transpose on the equation \eqref{NARE}, we know that $\Phi^T$ is the minimal nonnegative solution of the equation 
\begin{equation}\label{eqnz}
ZC^TZ-ZA^T-D^TZ+B^T=0, 
\end{equation}
and for the corresponding singular $M$-matrix 
\[
\breve{K}=\left [\begin{array}{cc} 
A^T & -C^T\\
-B^T & D^T\end{array} \right ],
\]
we have $\breve{K}\left[\begin{array}{c}
u_2\\
u_1
\end{array}\right ]=0$ by \eqref{uuvv}. If $\breve{K}$ were a regular $M$-matrix, we could apply the result 
$\Phi v_1\le v_2$  to the equation \eqref{eqnz} to get $\Phi^T u_2\le u_1$, or $u_2^T\Phi \le u_1^T$. 
However, $\breve{K}$ is not always regular when $K$ is, as seen from $K=\left [\begin{array}{cc} 
0 & 0\\
-1 & 1\end{array} \right ]$. 
To prove $\Phi^T u_2\le u_1$, we first assume that $D$ has no zero diagonal entries and proceed as in the proof of Lemma \ref{suff}. Let $A_1, A_2, D_1, D_2$ be as given there. Let $Z_0=0$ and consider the fixed-point iteration 
\[
D_1Z_{i+1}+Z_{i+1}A_1=Z_iC^TZ_i+Z_iA_2^T+D_2^TZ_i+B^T. 
\]
It is easily shown by induction that $Z_i\le Z_{i+1}\le \Phi^T$ for all $i\ge 0$. It follows that 
$\lim_{i\to \infty}Z_i=Z_*$ exists and is a nonnegative solution of \eqref{eqnz}. 
Since $Z_*\le \Phi^T$ and $\Phi^T$ is the minimal nonnegative solution, we have $Z_*=\Phi^T$. As in the proof of Lemma \ref{suff}, we have $Z_iu_2\le u_1$ for all $i\ge 0$. 
Thus $\Phi^T u_2\le u_1$. When $D$ has zero diagonal entries, we can still show that $\Phi^T u_2\le u_1$ using the procedure in the 
proof of Proposition \ref{pr}.

By applying the results $\Phi v_1\le v_2$ and $u_2^T\Phi \le u_1^T$   to the dual equation \eqref{nare2}, we have 
$\Psi v_2\le v_1$ and $u_1^T\Psi \le u_2^T$. 
 \end{proof}

\begin{thm}\label{3case}
Let $K$ be a regular singular $M$-matrix and let  $u_1, u_2, v_1, v_2$ be as in \eqref{uuvv}. Then 
\begin{enumerate}
\item[\rm (i)] If $u_1^Tv_1>u_2^Tv_2$, then $D-C\Phi$ is singular and $A-\Phi C$ is nonsingular, 
$\Phi v_1=v_2$, $\Psi v_2\ne v_1$, $u_1^T\Psi = u_2^T$, $u_2^T\Phi\ne u_1^T$. 
\item[\rm (ii)] If $u_1^Tv_1<u_2^Tv_2$, then $D-C\Phi$ is nonsingular and $A-\Phi C$ is singular, 
$\Phi v_1\ne v_2$, $\Psi v_2= v_1$,  $u_1^T\Psi \ne u_2^T$, $u_2^T\Phi= u_1^T$. 
\item[\rm (iii)] If $u_1^Tv_1=u_2^Tv_2$, then $D-C\Phi$ and $A-\Phi C$ are both singular, 
$\Phi v_1= v_2$, $\Psi v_2= v_1$,  $u_1^T\Psi = u_2^T$, $u_2^T\Phi= u_1^T$. 
\end{enumerate}
\end{thm}

\begin{proof} Rewriting \eqref{eqnH} as 
\[
H=\left [\begin{array}{cc}
I_n & 0\\
\Phi & I_m
\end{array} \right ] 
\left [\begin{array}{cc}
D-C\Phi & -C\\
0 & -(A-\Phi C)
\end{array} \right ]\left [\begin{array}{cc}
I_n & 0\\
-\Phi  & I_m
\end{array} \right ] 
\]
and using \eqref{uv}, we get 
\begin{equation}\label{u1}
(u_1^T-u_2^T\Phi)(D-C\Phi)=0, \quad   (A-\Phi C)(\Phi v_1-v_2)=0. 
\end{equation}
For the dual equation \eqref{nare2}, we have the corresponding result
\begin{equation}\label{u2}
(u_2^T-u_1^T\Psi)(A-B\Psi)=0, \quad   (D-\Psi B)(\Psi v_2-v_1)=0. 
\end{equation}
If $D-C\Phi$ is nonsingular, then $u_1^T=u_2^T\Phi$ by \eqref{u1} and $u_1^Tv_1=u_2^T\Phi v_1\le u_2^Tv_2$ by Lemma \ref{lem9}. 
If $A-B\Psi$ is nonsingular, then $u_2^T=u_1^T\Psi$ by \eqref{u2} and $u_2^Tv_2=u_1^T\Psi v_2\le u_1^Tv_1$ by Lemma \ref{lem9}. 

We now prove (i). When $u_1^Tv_1> u_2^Tv_2$, $D-C\Phi$ is singular, and then $A-\Phi C$ and $A-B\Psi$ are nonsingular by Lemma \ref{lem8}, Theorem \ref{thmev} and Corollary \ref{sim}. It follows from \eqref{u1} and \eqref{u2} that $\Phi v_1=v_2$ and $u_1^T\Psi = u_2^T$. 
We now show that $\Psi v_2\ne v_1$. This is trivial when $v_2=0$. Suppose $v_2\ne 0$ and $\Psi v_2=v_1$. Then 
$(A-B\Psi)v_2=Av_2-Bv_1=0$, contradictory to the nonsingularity of $A-B\Psi$. Finally, $u_2^T\Phi\ne u_1^T$ is trivial 
when $u_2=0$. Suppose $u_2\ne 0$ and  $u_2^T\Phi = u_1^T$. Then $u_2^T(A-\Phi C)=-u_1^T C+u_2^T A=0$, 
contradictory to the nonsingularity of $A-\Phi C$.

The results in (ii) are obtained by applying the results in (i) to the dual equation \eqref{nare2}
 and using Corollary \ref{sim}.

To prove (iii), we consider (as in \cite{guhi07})
\[
K(\alpha)=\left [ \begin{array}{cc} D & -C\\ -\alpha B & \alpha A \end{array} \right ]
\]
for $\alpha>1$. It is clear that $K(\alpha)$ is still a regular singular $M$-matrix, and the null spaces of $K(\alpha)$ and 
$K(\alpha)^T$ are still one-dimensional. We now have 
\[
u_1(\alpha)=u_1, \quad u_2(\alpha)=\alpha^{-1}u_2, \quad 
v_1(\alpha)=v_1, \quad v_2(\alpha)=v_2. 
\]
So we have $u_1(\alpha)^Tv_1(\alpha)>u_2(\alpha)^Tv_2(\alpha)$. 
Let $\Phi(\alpha)$ and $\Psi(\alpha)$ be the minimal nonnegative solutions of equations \eqref{NARE} and \eqref{nare2}, 
 respectively,  corresponding to $K(\alpha)$. 
We have by (i) that $D-C\Phi(\alpha)$ is singular, $\Phi(\alpha) v_1=v_2$, $u_1^T\Psi(\alpha) = \alpha^{-1}u_2^T$. 
Using the notation in Lemma \ref{lemev}, $H$  has $m_1$ eigenvalues in $\mathbb C_<$, $n_1$ eigenvalues in $\mathbb C_>$, and 
$r$ zero eigenvalues. By our assumption, $H$ has one $r\times r$ Jordan block 
associated with the zero eigenvalue. 
Let 
\[
H(\alpha)=\left [ \begin{array}{cc} D & -C\\ \alpha B & -\alpha A \end{array} \right ].
\]
For any $\alpha>1$ sufficiently close to $1$, $m-m_1$ zero eigenvalues of $H$ are perturbed to 
eigenvalues of $H(\alpha)$ in $\mathbb C_{<}$ and the real parts of these $m-m_1$ eigenvalues of $H(\alpha)$  are strictly larger than those of the remaining $m_1$ eigenvalues 
of $H(\alpha)$ in $\mathbb C_{<}$;  $n-n_1$ zero eigenvalues of $H$ are perturbed to eigenvalues of $H(\alpha)$ in $\mathbb C_{\ge}$, and the real parts of these $n-n_1$ eigenvalues of $H(\alpha)$  are strictly smaller than those of the remaining $n_1$ eigenvalues 
of $H(\alpha)$ in $\mathbb C_{>}$. 
By using the perturbation theory in section 16.5 of \cite{glr86} and section 5 of \cite{gr86}, 
we can modify the arguments leading to \cite[Theorem 3.3]{guhi07} to obtain that 
\[
\|\Phi(\alpha)-\Phi\|\le c \|K(\alpha)-K\|^{1/r}, \quad 
\|\Psi(\alpha)-\Psi\|\le c \|K(\alpha)-K\|^{1/r} 
\]
for some constant $c>0$. 
We note that we only need to adapt the arguments for case (a) in \cite{guhi07}, which 
are valid also when $K(\alpha)$ is singular.  
The arguments for case (b) in \cite{guhi07}  are to obtain stronger statements that are not needed here. 
We now have 
\[
\lim_{\alpha\to 1^+}\Phi(\alpha)=\Phi, 
\quad 
\lim_{\alpha\to 1^+}\Psi(\alpha)=\Psi. 
\]
Letting $\alpha\to 1^+$ in our conclusions about $\Phi(\alpha)$ and $\Psi(\alpha)$, we obtain that $D-C\Phi$ is singular, $\Phi v_1=v_2$, $u_1^T\Psi = u_2^T$. 
The other results in (iii) follows by duality. 
 \end{proof}

\begin{cor}
Let $K$ be a regular singular $M$-matrix and let  $u_1, u_2, v_1, v_2$ be as in \eqref{uuvv}. 
If $u_1^Tv_1=u_2^Tv_2$, then $I_m-\Phi \Psi$ and $I_n-\Psi\Phi$ are regular singular $M$-matrices. 
\end{cor}

\begin{proof}
By Proposition \ref{pr}, $I_m-\Phi \Psi$ and $I_n-\Psi\Phi$ are regular $M$-matrices. 
By Theorem \ref{3case}, $\Phi\Psi v_2=\Phi v_1=v_2$ and $\Psi \Phi v_1=\Psi v_2=v_1$. 
So $v_1$ and $v_2$ are both nonzero vectors, and $1$ is an eigenvalue of $\Phi \Psi$ and 
$\Psi \Phi$. Thus $I_m-\Phi \Psi$ and $I_n-\Psi\Phi$ are both singular. 
 \end{proof} 

We also have the following conjecture. 
\begin{conj}\label{conj1}
Let $K$ be a regular singular $M$-matrix and let  $u_1, u_2, v_1, v_2$ be as in \eqref{uuvv}. 
If $u_1^Tv_1\ne u_2^Tv_2$, then $I_m-\Phi \Psi$ and $I_n-\Psi\Phi$ are nonsingular $M$-matrices. 
\end{conj}

The conjecture is known to be true when $K$ is irreducible \cite{gim07}. Attempts to find counterexamples 
for the reducible case have failed and we are led to believe that the conjecture is true. But at present we can prove it only under a 
positivity assumption. 

\begin{prop}\label{pc}
Let $K$ be a regular singular $M$-matrix. Then $I_m-\Phi \Psi$ and $I_n-\Psi\Phi$ are nonsingular $M$-matrices 
if $u_1^Tv_1\ne u_2^Tv_2$ and at least one of $\Phi$ and $\Psi$ is positive. 
\end{prop}

\begin{proof}
Assume $u_1^Tv_1> u_2^Tv_2$. If $\Phi>0$, we know from Lemma \ref{lem9} and Theorem \ref{3case} that 
$\Phi \Psi v_2<\Phi v_1=v_2$. If $\Psi>0$, we know from Lemma \ref{lem9} and Theorem \ref{3case} that 
$u_2^T \Phi \Psi <u_1^T \Psi =u_2^T$. 
In either case we have $\rho(\Phi\Psi)=\rho(\Psi\Phi)<1$ (see \cite{bepl94}). So $I_m-\Phi \Psi$ and $I_n-\Psi\Phi$ are nonsingular $M$-matrices. The result for the case $u_1^Tv_1< u_2^Tv_2$ follows by duality.
 \end{proof}

As we mentioned earlier, $\Phi$ and $\Psi$ are always positive when $K$ is irreducible. 
When $K$ is reducible, they may still be positive. We also have the following related result. 

\begin{prop}\label{2p}
Let $K$ be a regular $M$-matrix. Then the equation \eqref{NARE} 
 can have at most two positive solutions if the eigenvalues of $H$ are all simple. 
\end{prop}

\begin{proof}
Let $w>0$ be such that $Kw\ge 0$, and let $K_{\epsilon}=K+\epsilon w^Tw I-\epsilon ww^T$. 
Then $K_{\epsilon}w=Kw\ge 0$. So $K_{\epsilon}$  is an irreducible $M$-matrix for any $\epsilon>0$. 
Suppose that the equation \eqref{NARE} (corresponding to $K$) has more than two positive solutions. 
Then each of these solutions is determined by $n$ eigenvectors of $H$ corresponding to $n$ of its simple eigenvalues. 
So each solution changes continuous as $K$ is perturbed to $K_{\epsilon}$. It then follows that the equation \eqref{NARE} 
corresponding to $K_{\epsilon}$ would have more than two positive solutions for $\epsilon>0$ sufficiently small, which is contradictory to 
\cite[Theorem  2.13]{bim12}. 
 \end{proof}

\section{Numerical methods for the minimal nonnegative solution}

When the matrix $K$ in \eqref{Mm} is a nonsingular $M$-matrix or an irreducible singular $M$-matrix, various numerical methods are available to find the minimal nonnegative solution of the equation \eqref{NARE}. See \cite{bgx06,bilm05,bim12,bmp10,gaba11,guo01,guo06,guo12,guhi07,gim07,guba05,glx06,iapo12,wwl12,xxl12}. 
In this section we assume that $K$ is a regular $M$-matrix that is singular and reducible. 
We have already seen form Theorem \ref{thmn} that the case where $A$ has zero diagonal entries can be reduced to the case where 
$A$ has no zero diagonal entries, and we know from Lemma \ref{suff}  that the minimal nonnegetive solution of the reduced equation 
can be found by a simple fixed-point iteration. However,  some othe fixed-point iterations discribed in \cite{guo01} may not be well-defined 
when the matrix $I\otimes A+ D^T \otimes I$ is a singular $M$-matrix, where $\otimes$ is the Kronecker product, even if all diagonal entries 
of $A$ and $D$ are positive. 
We now assume that the matrix $H$ in \eqref{Hm} has a simple zero eigenvalue. Recall that zero is a simple eigenvalue of $H$ if and only if  $u_1^Tv_1\ne u_2^Tv_2$   for the vectors $u_1, u_2, v_1, v_2$ in \eqref{uuvv}. We then know from Theorem \ref{3case} that the matrix 
$I\otimes (A-\Phi C)+(D-C\Phi)^T \otimes I$ is a nonsingular $M$-matrix. 
It follows that the matrix $I\otimes A+ D^T \otimes I$ is also  a nonsingular $M$-matrix. 

The class of fixed-point iterations discussed in \cite{guo01} can then  be applied to find the minimal nonnegative solution 
$\Phi$. The theoretical results in \cite{guo01} (Theorems 2.3, 2.4, 2.5, 2.7) are all applicable. The convergence of these 
fixed-point iterations is linear. 

Since zero is  a simple eigenvalue of $H$, we have $\lambda_n>\lambda_{n+1}$ in Theorem \ref{thmev} and the minimal nonnegative solution $\Phi$ 
can be computed by finding the invarant subspace of $H$  corresponding to the eigenvalues $\lambda_1, \lambda_2, \ldots, 
\lambda_n$ using the Schur method, as in \cite{guo01}. 
When the matrix $K$ is such that $Ke=0$, as in the study of Markov chains, the eigenvector corresponding to the zero eigenvalue is known 
exactly and the modified Schur method discribed in \cite{guo06} can be used to find $\Phi$ with much better accuracy when 
$\lambda_n\approx \lambda_{n+1}$. 

The minimal nonnegative solution $\Phi$ can also be found by Newton's method. Indeed, Theorem 2.3 of \cite{guhi07}  is still valid. 
In \cite{guhi07},  $I\otimes A+ D^T \otimes I$ is a nonsingular $M$-matrix 
because $K$ is a nonsingular $M$-matrix or an irreducible singualr $M$-matrix. Here, $I\otimes A+ D^T \otimes I$ is a nonsingular $M$-matrix 
because we assume that zero is  a simple zero eigenvalue of $H$. The proof of \cite[Theorem 2.3]{guhi07} for the case where $K$ is a nonsingular $M$-matrix is valid for the current situation since $I\otimes (A-\Phi C)+(D-C\Phi)^T \otimes I$ is a nonsingular $M$-matrix by 
our assumption on $H$. The convergence of Newton's method is quadratic. 

A structure-preserving doubling algorithm (SDA) has been presented in \cite{glx06} to find the minimal nonnegative solutions $\Phi$ and $\Psi$ simultaneously 
when $K$ is a nonsingular $M$-matrx. It is easy to see that the main results in \cite{glx06} (Theorems 3.1 and 4.1) 
still hold when $K$ is a regular singular $M$-matrix. One would only need to modify some statements in their proofs. 
Specifically, we now have that the matrices $R=D-C\Phi$ and $S=A-B\Psi$ are  regular $M$-matrices 
(so all eigenvalues of $R$ and $S$ are in $\mathbb C_{\ge}$), the matrices $D$ and $A$ are regular $M$-matrices, 
and $\rho(R_{\gamma})\le 1, \rho(S_{\gamma})\le 1$ for all $\gamma>0$, where 
$R_{\gamma}=(R+\gamma I_n)^{-1}(R-\gamma I_n)$ and $S_{\gamma}=(S+\gamma I_m)^{-1}(S-\gamma I_m)$. 
The SDA in \cite{glx06} generates four sequences $\{E_k\}, \{F_k\}, \{G_k\}, \{H_k\}$ and \cite[Theorem 4.1]{glx06} implies that 
\begin{equation}\label{crate}
\limsup_{k\to \infty}\sqrt[2^k]{\|H_k-\Phi\|}\le \rho(R_{\gamma})\rho(S_{\gamma}), 
\quad \limsup_{k\to \infty}\sqrt[2^k]{\|G_k-\Psi\|}\le \rho(R_{\gamma})\rho(S_{\gamma})
\end{equation}
for $\gamma>\max\{\max a_{ii}, \max d_{ii}\}$. 

The SDA of \cite{glx06} is further studied in \cite{gim07}. When $K$ is an irreducible singular $M$-matrix,  it is shown in \cite{gim07} 
that the SDA is well defined and \eqref{crate} holds for  
$\gamma\ge \max\{\max a_{ii}, \max d_{ii}\}$. Moreover, among these $\gamma$ values, 
the SDA will have fastest convergence for $\gamma = \max\{\max a_{ii}, \max d_{ii}\}$. 
The proof in \cite{gim07} is based on the fact that $0\le G_k<\Psi$, $0\le H_k<\Phi$, and 
$\rho(G_kH_k)<\rho(\Psi\Phi)\le 1$. 
When $K$ is a nonsingular $M$-matrix, we can reach the same conclusion about the SDA since we now have 
$0\le G_k\le \Psi$, $0\le H_k\le \Phi$, and $\rho(G_kH_k)\le \rho(\Psi\Phi)<1$. 

When $K$ is a reducible singular $M$-matrix and $H$ has a simple zero eigenvalue, we know 
(from  Lemma \ref{lem8}, Theorem \ref{3case} and Corollary \ref{sim})  
that one of 
$\rho(R_{\gamma})$ and $\rho(S_{\gamma})$ is equal to $1$ and the other is strictly smaller than $1$. 
It follows that $H_k\to \Phi$ and $G_k\to \Psi$, both quadratically. 
We can allow $\gamma = \max\{\max a_{ii}, \max d_{ii}\}$ if we have $\rho(\Phi\Psi)<1$. Note that this special $\gamma$ value is positive since $K$ can have at most one zero diagonal entry when $H$ has a simple zero eigenvalue. 
Note also that we do have $\rho(\Phi\Psi)<1$ if Conjecture \ref{conj1} is true. 

A negative answer to Conjecture \ref{conj1} is also interesting. Suppose we can find an example with $\rho(\Phi\Psi)=1$ 
when $K$ is a reducible singular $M$-matrix and $H$ has a simple zero eigenvalue. Then for this example, 
the matrices to be inverted in each step of the SDA, $I_n-G_kH_k$ and $I_m-H_kG_k$, will be nearly singular 
near convergence. As a result, we would encounter numerical difficulities with the SDA even though it has 
quadratic convergence in exact arithmetic. Moreover, the SDA would also run into difficulties for nonsingular $M$ -matrices 
or irreducible singular $M$-matrices that are close to this particular reducible singular $M$-matrix. 
Should these situations occur, one would be forced to use Newton's method, whose computational work is roughly 3 times that for the SDA in each iteration. 

In the SDA of \cite{glx06}, the parameter $\gamma$ is used in a Cayley transform applied to the matrix $H$. 
In \cite{wwl12}, a generalized Cayley transform involving two parameters $\alpha$ and $\beta$ is applied to $H$, and the resulting doubling algorithm is called ADDA. 
The iteration formulas  in the ADDA are exactly the same as in \cite{glx06}, but the initial matrices $E_0, F_0, G_0, H_0$ for ADDA 
($G_0$ and $H_0$ are denoted by $Y_0$ and $X_0$ in \cite{wwl12}) are determined by 
\begin{eqnarray}\label{poloni}
\left [\begin{array}{cc}
E_0 & -G_0\\
-H_0  & F_0
\end{array} \right ]&=&\left [\begin{array}{cc}
D+\alpha I_n & -C\\
B & -A-\beta I_m
\end{array} \right ]^{-1} \left [\begin{array}{cc}
D-\beta  I_n & -C\\
B & -A+\alpha I_m
\end{array} \right ]\\
&=&\left [\begin{array}{cc}
D+\alpha I_n & -C\\
-B & A+\beta I_m
\end{array} \right ]^{-1} \left [\begin{array}{cc}
D-\beta  I_n & -C\\
-B & A-\alpha I_m
\end{array} \right ]. \nonumber 
\end{eqnarray}
The above compact form for determining $E_0, F_0, G_0, H_0$ is based on a result of Poloni \cite{polo10} 
(see also \cite[Theorem 5.5]{bim12}). 
It is easy (but tedious) to verify directly that these initial matrices are the same as those determined in \cite{wwl12}. 
When $\alpha=\beta=\gamma$, the ADDA is reduced to the SDA and the initial matrices $E_0, F_0, G_0, H_0$ from 
\eqref{poloni} are exactly the same as those in \cite{glx06}. 

When $K$ is a nonsingular $M$-matrix or an irreducible singular $M$-matrix, the convergence theory of the ADDA has been given in \cite{wwl12}. The following result allows $K$  to be singular and reducible. 

\begin{thm}\label{thm99}
Let $K$ be a regular singular $M$-matrix and assume that 
$\alpha> \max a_{ii}$ and $\beta> \max d_{ii}$. Then the ADDA is well defined with 
$I-G_kH_k$ and $I-H_kG_k$ being nonsingular $M$-matrices for each $k\ge 0$. 
Moreover, $E_0\le 0, F_0\le 0$,  $E_k\ge 0, F_k\ge 0$, 
$0\le H_{k-1}\le H_k\le \Phi$, $0\le G_{k-1}\le G_k\le \Psi$ for all $k\ge 1$, 
and 
\[
\limsup_{k\to \infty}\sqrt[2^k]{\|H_k-\Phi\|}\le r(\alpha, \beta), 
\quad 
 \limsup_{k\to \infty}\sqrt[2^k]{\|G_k-\Psi\|}\le r(\alpha, \beta),  
\]
where 
$r(\alpha, \beta)=
\rho\left ((R+\alpha I)^{-1}(R-\beta I)\right )\cdot \rho\left ((S+\beta I)^{-1}(S-\alpha I)\right )$ 
with $R=D-C\Phi, S=A-B\Psi$. 
\end{thm}

\begin{proof} 
The proof is largely the same as in \cite{wwl12}. The main difference is that the approach in 
\cite{gim07} cannot be used to prove that $I-G_kH_k$ and $I-H_kG_k$ are nonsingular $M$-matrices for each $k\ge 0$. Instead, this property is proved  using  the approach in \cite{glx06}, 
which does not allow $\alpha= \max a_{ii}$ or $\beta= \max d_{ii}$. 
 \end{proof} 

When $H$ has a simple zero eigenvalue, we know from Lemma \ref{lem8}, Theorem \ref{3case}, and Corollary \ref{sim} that 
one of the matrices $R$ and $S$ is singular and the other is nonsingular. It follows from \cite[Theorem 2.3]{wwl12} that 
$r(\alpha, \beta)<1$ in Theorem \ref{thm99}. Thus, $H_k$ converges to $\Phi$ quadratically, and $G_k$ converges to $\Psi$ quadratically. 
Since $K$ has at most one zero diagonal entry when $H$ has a simple zero eigenvalue, we have $\max a_{ii}>0$ and $\max d_{ii}>0$ 
when $m, n\ge 2$ (which we shall assume). By \cite[Theorem 2.3]{wwl12} $r(\alpha, \beta)$ is strictly increasing in $\alpha$ for 
$\alpha\ge \max a_{ii}$, and is strictly increasing in $\beta$ for $\beta\ge \max d_{ii}$. 
If Conjecture \ref{conj1} is true, then $\rho(\Phi \Psi)<1$ when $H$ has a simple zero eigenvlaue. In this case we can use the approach in \cite{gim07} to prove that the ADDA is well defined with 
$I-G_kH_k$ and $I-H_kG_k$ being nonsingular $M$-matrices for each $k\ge 0$ even when 
$\alpha= \max a_{ii}$ and $\beta= \max d_{ii}$. 
Therefore, if Conjecture \ref{conj1} is true, we should normally use the optimal values $\alpha= \max a_{ii}$ and $\beta= \max d_{ii}$ for the ADDA. 
Otherwise, we may have to take $\alpha$ and $\beta$ to be slightly larger. 
When $\max a_{ii}$ is much larger (or smaller) than $\max d_{ii}$, the convergence of ADDA may be significantly faster than that of the SDA 
in \cite{glx06}, it is also faster than the SDA-ss in \cite{bmp10}. When $\max a_{ii}\approx \max d_{ii}$ but
$\max a_{ii}\gg \min a_{ii}$ and/or  $\max d_{ii}\gg \min d_{ii}$, all three doubling algorithms take roughly the same number of iterations for convergence,
which may be significantly larger than that required for the Newton iteration, as noted in \cite{guo12}. 

We also remark that the modified Schur method described earlier can achieve very good normwise accuracy, but this method is unable to 
compute any tiny positive entries in the minimal nonnegative solution with good relative accuracy, while the fixed-point iterations, the Newton iteration, the doubling algorithms can be implemented in such a way that these tiny positive entries can be computed with high relative accuracy \cite{xxl12}.

\section{Some illustrating examples}

When the matrix $K$ in \eqref{Mm} is a regular $M$-matrix that is singular and reducible and the matrix $H$ in \eqref{Hm} has a simple zero eigenvalue, the numerical algorithms 
reviewed in the previous section can be applied without difficulty (with the help of our new theoretical results) and the numerical behaviour of these algorithms is very similar to that for 
the equation \eqref{NARE} corresponding to a nearby nonsingular $M$-matrix or irreducible singular $M$-matrix. So in this section, we only present a few examples of small sizes to  illustrate some of the theoretical  results established in this paper. 

Take 
\[
K=\left [\begin{array}{cccc}
2 & -1 & -1 & 0\\
0 & 2 & -1 & -1\\
0 & -1 & 2 & -1\\
0 & -1 & -1 & 2
\end{array} \right ]. 
\]
Then $K$ is a  regular $M$-matrix that is singular and reducible, with $Ke=0$.  
The $2\times 2$ matrices $A, B, C, D$ are then determined by \eqref{Mm}. 
We find that the eigenvalues of $H$ are $2, 1, 0, -3$. By Theorem \ref{thmev}
the minimal nonnegative solution $\Phi$ of \eqref{NARE} is determined by the eigenvectors corresponding to 
the eigenvalues $2$ and $1$, and is found to be 
\[
\left [\begin{array}{cc}
0 & 1/2\\
0 & 1/2
\end{array} \right ]. 
\]
As shown in Lemma \ref{lem9} and Theorem \ref{3case} (ii) , $\Phi$ is substochastic. 

We now take 
\[
K=\left [\begin{array}{cccc}
2 & -1 & 0 & -1\\
-1 & 2 & 0 & -1\\
0 & 0 & 2 & -2\\
-1 & -1 & 0 & 2
\end{array} \right ]. 
\]
Then $K$ is a  regular $M$-matrix that is singular and reducible, with $Ke=0$.  
The $2\times 2$ matrices $A, B, C, D$ are again determined by \eqref{Mm}. 
We find that the eigenvalues of $H$ are $3, 0, -1, -2$. 
The minimal nonnegative solution $\Phi$ of \eqref{NARE} is determined by the eigenvectors corresponding to 
the eigenvalues $3$ and $0$, and is found to be 
\[
\left [\begin{array}{cc}
1/2 & 1/2\\
1/2 & 1/2
\end{array} \right ]. 
\]
As shown in Theorem \ref{3case} (i), $\Phi$ is stochastic. By Proposition \ref{2p} the equation 
\eqref{NARE} has at most two positive solutions. For this example, \eqref{NARE} has 
exactly two positive solutions. As suggested by \cite[Lemma 11]{figu06}, 
the other positive solution of \eqref{NARE} is determined by the eigenvectors corresponding to 
the eigenvalues $3$ and $-1$, and is found to be 
\[
\left [\begin{array}{cc}
2 & 2\\
1 & 1
\end{array} \right ]. 
\]
For this example, Proposition 
\ref{pc} also applies, and we know that $I-\Phi\Psi$ and $I-\Psi\Phi$ are nonsingular $M$-matrices. 

Finally we take 
\[
K=\left [\begin{array}{cccc}
1 & 0 & 0 & -1\\
0 & 1 & 0 & -1\\
0 & 0 & 1 & -1\\
0 & -1 & 0 & 1
\end{array} \right ]. 
\]
Then $K$ is a  regular $M$-matrix that is singular and reducible, with $Ke=0$.  
The $2\times 2$ matrices $A, B, C, D$ are again determined by \eqref{Mm}. 
We find that the eigenvalues of $H$ are $1, 0, 0, -1$ and that there is only one linearly independent eigenvector corresponding to the double eigenvalue $0$. 
The minimal nonnegative solution of \eqref{NARE} is determined by the eigenvectors corresponding to 
the eigenvalues $1$ and $0$, and is found to be 
\[
\left [\begin{array}{cc}
0 & 1\\
0 & 1
\end{array} \right ].  
\]
It is easy to verify that this is the only nonnegative solution of \eqref{NARE}. For this example, we have case (iii) in Theorem \ref{3case} and 
$\Phi$ is stochastic. 

\section{Conclusions} 
In this paper we have studied  algebraic Riccati  equations associated with regular $M$-matrices. Our results extend previous results for 
algebraic Riccati  equations associated with nonsingular $M$-matrices or irreducible singular $M$-matrices. In the future, we can use the term 
$M$-matrix algebraic Riccati equation to refer to equation \eqref{NARE} for which the matrix $K$ in \eqref{Mm} is a regular $M$-matrix. 
While we have been able to prove a number of theoretical results for this larger class of Riccati equations (sometimes with proofs simpler than 
previous proofs for special cases; see proof of Theorem \ref{thmev} for example), we are unable to prove the statement given in Conjecture \ref{conj1}. We have pointed out that an answer to the conjecture (whether it is positive or negative) has interesting implications. 
Some further research is needed to settle this conjecture.

\end{document}